\documentclass[12pt,reqno]{amsart}

\usepackage{amsmath}
\usepackage{amsthm}
\usepackage{amscd}
\usepackage{amssymb}
\usepackage{graphicx}
\usepackage{epsfig}
\theoremstyle{plain}
\newtheorem{theorem}{Theorem}[section]
\newtheorem{proposition}[theorem]{Proposition}
\newtheorem{lemma}[theorem]{Lemma}
\newtheorem{corollary}[theorem]{Corollary}

\newtheorem{definition}[theorem]{Definition}
\newtheorem{remark}[theorem]{Remark}

\def\int{\text{interior}}
\def\hyp {\hbox {\rm {H \kern -2.8ex I}\kern 1.25ex}}
\def\reals {\hbox {\rm {R \kern -2.8ex I}\kern 1.15ex}}
\def\integers {\hbox {\rm { Z \kern -2.8ex Z}\kern 1.15ex}}
\def\naturals {\hbox {\rm {N \kern -2.8ex I}\kern 1.20ex}}
\def\rationals {\hbox {\rm { Q \kern -2.2ex l}\kern 1.15ex}}
\def\hyp {\hbox {\rm {H \kern -2.7ex I}\kern 1.25ex}}

\def\strutdepth{\dp\strutbox}
\def \ss{\strut\vadjust{\kern-\strutdepth \sss}}
\def \sss{\vtop to \strutdepth{
\baselineskip\strutdepth\vss\llap{$\diamondsuit\;\;$}\null}}

\def\strutdepth{\dp\strutbox}
\def \sst{\strut\vadjust{\kern-\strutdepth \ssss}}
\def \ssss{\vtop to \strutdepth{
\baselineskip\strutdepth\vss\llap{$\spadesuit\;\;$}\null}}

\def\strutdepth{\dp\strutbox}
\def \ssh{\strut\vadjust{\kern-\strutdepth \sssh}}
\def \sssh{\vtop to \strutdepth{
\baselineskip\strutdepth\vss\llap{$\heartsuit\;\;$}\null}}

\begin{document}

 \title{ High distance Heegaard splittings via fat train tracks}

\author{Martin Lustig} 
 
\author{Yoav Moriah}

\thanks {This research is supported by  a grant from the High Council 
for Scientific and Technological Cooperation between France and 
Israel}

\begin{abstract}
We define {\it fat} train tracks and use them to give a combinatorial 
criterion for the Hempel distance of Heegaard splittings for closed 
orientable $3$-manifolds.  We apply this criterion to 3-manifolds 
obtained from surgery on knots in $S^{3}$.
\end{abstract}


\address{Math\'ematiques (LATP)\\
Universit\'e P. C\'ezanne - Aix-Marseille III \\
Ave. Escad. Normandie-Niemen, 13397 Marseille 20, France}
\email{martin.lustig@univ-cezanne.fr}

\address{Department of Mathematics \\
Technion \\
Haifa, 32000 Israel}
\email{ymoriah@tx.technion.ac.il}


\maketitle

\section{Introduction}\label{introduction}

In this paper we derive some combinatorial tools which are useful in
determining estimates for the distance of Heegaard splittings of  
closed
oriented 3-dimensional manifolds.  In particular, we prove:

\vskip8pt

\noindent
{\bf Theorem \ref{secondsurgery}'.} 
{\it  For any $n \geq 0$ there   exist a knot $k \in S^{3}$ and an integer $m_{0}$, such that for  any
$m \in \integers \smallsetminus \{m_{0}, m_{0}+1, m_{0} +   2, m_{0}+3\}$ the $3$-manifold  $M(m)$,
obtained by $\frac {1}{m}$-surgery along a horizontal slope on $k$,  admits a  Heegaard splitting 
$ M(m) = V \cup_{\Sigma} W$   which satisfies $$d(V, W) \geq n \, .$$ }

Here $V \subset M(m)$ and $W \subset M(m)$ are handlebodies which  intersect in 
$\partial V = \Sigma = \partial W$.  Each of them, $V$ and $W$, determines an infinite subcomlex 
$\mathcal{D}(V)$ and $\mathcal{D}  (W)$ in the curve complex $\mathcal{C}(\Sigma)$, and the 
distance $d(V, W)$ is defined as minimal distance in $\mathcal{C}(\Sigma)$ between a vertex  
$D \in \mathcal{D}(V)$ and a vertex $E \in \mathcal{D}(W)$. The curve complex $\mathcal{C}  (\Sigma)$
and its metric $d$ are defined at the beginning of Section 2.

\vskip 10pt

The main combinatorial tool, to prove this result, is the following   theorem, which is derived
from some basic observations that can already be found in the work of   Masur-Minsky \cite{MM},
such as in particular the use of nested train track towers.

\vskip 10pt

\noindent
{\bf Theorem \ref{HSdistance}.}
{\it Let $M$ be an oriented $3$-manifold with a Heegaard splitting $M   = V \cup_{\Sigma} W$.
Consider complete decomposing systems $\mathcal D \in \mathcal{CDS}(V) $ and
$\mathcal E \in \mathcal{CDS}(W)$ which do not have waves with  respect to each other.
Let $\tau \subset \Sigma$ be a complete fat train track with   exceptional fibers
$\mathcal{E}_{\tau} = \mathcal E$, and assume that $\mathcal D$ is   carried by $\tau_{n}$,
for some $n$-tower of derived train tracks $\tau = \tau_{0} \supset \tau_{1} \supset \ldots \supset \tau_{n}$  
with $n \geq 2$.  Then the distance of the  given Heegaard splitting satisfies: 
$$d(V, W) \geq n$$ }

The technical terms of this theorem are explained in detail in   Sections 2, 3 and 4.
A innovative  feature of this paper are {\it fat} train tracks   (compare Definitions   \ref{fatraintrack} 
and \ref{completefatraintrack}), which allow a  smooth transition from disk systems in a handlebody to the train track technology needed to   obtain lower distance bounds in the curve complex.

\vskip5pt

This train track technology, presented below in Section 2, is based on our concept of a {\it derived} train track, which is a combinatorial   version of Masur-Minsky's ``Basic observation'' (\cite{MM}, \S 3.1) that turns out to be very handy in our context.

\vskip5pt

Special effort was made to give all proofs of this paper from basic notions.

\vskip8pt

\noindent
{\it Acknowledgments:}
The authours would like to thank Yair Minsky and Saul Schleimer for some very helpful comments, in particular for gently pointing out to us that some of our original efforts were already known to the  
experts in the field.

\section{Train tracks and distance in the Curve Complex}{\label{crvcmplx}

\vskip4pt

Recall that for an orientable connected surface $\Sigma$ of  genus  $g \geq 2$ the curve 
complex $\mathcal {C}(\Sigma)$  is defined as follows:
\begin{enumerate}
\item The set of vertices $\mathcal{C}^0(\Sigma)$ is the set of  isotopy classes of simple closed 
curves on $\Sigma$.
\item  An $n$-simplex is a collection $\{v_0, ..., v_{n}\}$ of   vertices which have mutually disjoint representative curves. 
\end{enumerate}

On the $1$-skeleton of $\mathcal {C}^1 (\Sigma)$ we define a metric   $d_{\mathcal{C}}( \cdot,\cdot)$  
by declaring the length of every edge to be $1$. For the purpose of  this paper it will suffice to consider 
only  $\mathcal {C}^1 (\Sigma)$. 

\vskip5pt

The curve complex $\mathcal{C}(\Sigma)$ was invented by Harvey ~\cite{Har} and has been 
the object of intense research in recent years (see \cite{AS}, \cite{He}, \cite{Kl}, \cite{Ko}, 
\cite{MM}, \cite{MM1}, \cite{MMS}, \cite{Sc}, just to mention a few from a much longer list). Its  importance is highlightened by the  fact that this complex, on  which the mapping class group  $MCG(\Sigma)$ 
acts naturally (but not properly, as  $\mathcal{C}(\Sigma)$ is not locally compact), is indeed a  
$\delta$-hyperbolic space in the sense of Gromov (see ~\cite{MM}). For background and more information  on the curve complex see ~\cite{Sc}.

\vskip10pt 

We now briefly review {\it train tracks} from a somewhat non-standard  perspective which will be 
used in the next two Sections:

\vskip10pt

A {\it train track} $\tau$ in a closed surface $\Sigma$ is a closed  subsurface with a 
{\it singular $I$-fibration}. By this we mean that  the interior of $\tau$ is fibered by open arcs  (each  homeomorphic to the interior of the unit interval $I$) and   that the  fibration extends to a fibration of the closed surface  $\tau$  by  properly  embedded closed arcs, except for finitely many {\it singular points  
on $\partial \tau$},  where precisely two  fibers meet. We call these  fibers  {\it singular fibers}. We admit the case that a fiber is {\it  doubly  singular}, i.e., both of  its endpoints  are singular. 

Two singular fibers are {\it adjacent} if they share a singular   point  as a common endpoint. A maximal connected union of singular  or  doubly singular $I$-fibers is called an {\it exceptional fiber}.  It  consists 
of a sequence of adjacent fibres, and hence it is either  homeomorphic  to a closed  interval, or to a simple closed curve on  $\Sigma$.  In  the latter case it will be called a {\it  cyclic  exceptional fiber}. 
We explicitely admit this second case, although we are aware of the fact that in the classical  train track literature this case is sometimes   suppressed.

We picture a singular point $P \in \partial \tau$ in such a way that   two arcs from $\partial \tau$ which intersect in $P$ converge towards  $P$ from the ``same direction", thus giving  rise to a {\it cusp point} 
on the boundary of the corresponding complementary component  of  $\tau$ in $\Sigma$.

We define the {\it type} of a complementary component $\Delta$ of   $\tau$ in  $\Sigma$ as given by 
the genus of $\Delta$ and the number of cusp points on its boundary.   If $\Delta$ is simply connected, we speak of  an $n$-gon if there are  precisely $n$ cusp points on $\partial  \Delta$.  For example, if 
$\partial \Delta$ contains precisely three  cusp points, we say that  $\Delta$ is a {\it triangle}. An arc of  
$\partial \Delta$ which joins two adjacent cusp points is called a    {\it side} of $\Delta$.  If all complementary regions  of the train track $\tau$ are triangles we  say that  $\tau$ is {\it maximal}.

\vskip8pt

\begin{remark} \label{traintrack} \rm 
Contracting every $I$-fiber to  a point defines a deformation of    $\tau$ to a graph  $\Gamma_{\tau}$. 
The exceptional fibers will be  deformed to the vertices of  $\Gamma_{\tau}$. The $I$-fiber structure  
on $\tau$  defines a gate structure on $\Gamma_{\tau}$,  which is  usually  visualized by giving 
$\Gamma_{\tau}$ the structure of  a branched  $1$-manifold.

The deformation $\tau \rightarrow \Gamma_{\tau}$ is a homotopy   equivalence only if there are no
cyclic exceptional fibers. Also, $\Gamma_{\tau}$ can be embedded as   a  retract into $\tau$ only if 
every exceptional fiber is {\it simple}, i.e.,  it consists of    exactly two singular $I$-fibers. In this case the corresponding   vertex of $\Gamma_{\tau}$ can be  mapped to the unique cusp point  $P$  contained in  
the exceptional fiber. Notice that even in this case the subsurface   $\tau \subset \Sigma$ fails to be a regular neighborhood of the   embedded graph $\Gamma_{\tau} \subset \Sigma$,  exactly at the 
singular  points of $\partial \tau$.

\end{remark}

\vskip8pt

\begin{remark} \label{simplyconnected}\rm
If $\tau$ is connected and has at least one exceptional fiber, then any closed curve $c \subset \tau$ which does not meet any of  the  exceptional fibers of $\tau$ is  contractible in $\tau$.  This 
follows directly from the fact that for  the above homotopy  equivalence $\tau \rightarrow \Gamma_{\tau}$ the image of $c$ is  a loop in  the graph $\Gamma_{\tau}$ that is disjoint from all  vertices of  $\tau$.
Note that $\Gamma_{\tau}$ is connected and  topologically different from $S^{1}$, by the assumption on  
$\tau$.

\end{remark}

\vskip8pt

An arc or a closed curve   in $\Sigma$ {\it is carried} by a  train track  $\tau \subset \Sigma$ if it
is contained in $\tau$ and is always transverse to the  $I$-fibers in $\tau$.  A simple closed curve 
{\it can be carried}  by  $\tau$ if it can be isotoped in $\Sigma$ to a curve that is carried   by $\tau$.
The two notions are not always being kept strictly apart, as most of the  time one is interested in curves only up to isotopy on $\Sigma$.

Two simple arcs or curves carried by $\tau$ are {\it parallel} if  they intersect  the same $I$-fibers, 
and these intersections occur on the two arcs (or curves) in   precisely the same order.

\vskip8pt

\begin{definition}\label{covers} \rm Let $\tau$ be a train track in 
a  surface $\Sigma$. An arc or a closed curve  on $\Sigma$ which is  carried by $\tau$ is  said to 
{\it cover} $\tau$ if  it meets every $I$-fiber of $\tau$.

\end{definition}

\vskip8pt

The statement of Lemma \ref{disjointcarried} below has been first observed in \cite{MM}, Section  3.1, 
``A basic observation''.  We include here a detailed proof of  this and the preceding Lemma 
\ref{entireside}, as we  will return later, in Section 3, to the same argument, which will be 
generalized (see the proof of Lemma \ref{disjointcarried2}).

\vskip8pt

\begin{lemma}\label{entireside}  
Let $\tau \subset \Sigma$ be a train track, and let $\mathcal{D}$ be  a finite collection of pairwise disjoint simple closed curves  which  together cover $\tau$. Then for every complementary component 
$\Delta$ of $\tau$ in $\Sigma$  and every side $\delta$ of $\Delta$  there is an arc $\alpha$ on some curve  of  $\mathcal D$ that runs parallel along all of $\delta$.

\end{lemma}

\vskip10pt

\begin{proof}
Let  $\sigma$ be an $I$-fiber of $\tau$ that meets $\delta$  in a  point  $P \in \partial \sigma \cap \delta$, and let $Q$ be the point  of $\sigma \cap  {\mathcal D}$ located on $\sigma$ closest to $P$. 
Let $\alpha$  be the maximal arc of $\mathcal D$ that contains $Q$   and runs  parallel to a subarc of 
$\delta$.  If $\alpha$ does not run parallel along all of $\delta$, it branches off at some singular $I$-fiber 
$\sigma'$ which has a (non-singular) endpoint on $\delta$.  Since $\mathcal D$ covers  $\tau$, there must  be another arc $\beta$ of  $\mathcal D$ that meets  $\sigma'$. But then the elongation of 
$\beta$, in the  direction of $\alpha$ towards $\sigma$, will proceed  between $\alpha$ and $\delta$, and thus give an intersection point with $\sigma$  which lies between $P$ and 
$Q$, contradicting the choice  of  $Q$. Thus $\alpha$ cannot  branch off from $\delta$ and hence 
runs  parallel along $\delta$ all  the way.

\end{proof}

\vskip8pt

\begin{lemma}\label{disjointcarried} 
Let $\tau$ be a maximal train track on  a  surface $\Sigma$.  Let   $\mathcal{D}$ be a  finite collection of pairwise disjoint simple  closed curves which together  cover $\tau$, and let  $D$ be an 
essential simple closed curve disjoint from $\mathcal{D}$. Then  $D$ can be  carried by $\tau$. 

\end{lemma}

\vskip8pt

\begin{proof}  

We can isotope $D$ off all complementary components of $\tau$  so   that, after the isotopy, the image 
of $D$ in $\tau$  will still be  disjoint from $\mathcal D$.  This   can be done at the expense of accepting 
that at finitely many singular points the curve $D$  is not  transverse to the   $I$-fibering. We can  now shorten $D$ at those  singularities and  thus  eventually eliminate  all of them except  those that occur in 
the  neighborhood of  a cusp  point $P$ of some complementary   component  $\Delta$.  In this  case there   are two arcs $\alpha$ and  $\alpha'$ on $D$ adjacent to such a  singularity $X \in D$, which run 
parallel to initial segments of the  two sides $\delta$  and  $\delta'$ of $\Delta$ that are adjacent to  $P$. 
Choose such a singularity $X$ and  two such arcs $\alpha$ and  $\alpha'$ which are closest to  $\delta$  and $\delta'$   respectively.  Notice that it is impossible that there is a curve from $\mathcal D$ 
that passes between $\alpha$ and $\delta$ or  between $\alpha'$ and $\delta'$, as the singularity $X$ lies on  both, $\alpha$ and $\alpha'\, $:  Since $D$ is disjoint from  $\mathcal D$ and the latter is carried by $\tau$, $X$ cannot be separated from the cusp point $P$ by  curves from $\mathcal D$.

By Lemma  \ref{entireside}, applied to $\mathcal D$, the arcs
$\alpha$ and $\alpha'$ can  be elongated to arcs $\hat \alpha$  and  $\hat \alpha'$ of $D$ that 
are  parallel to the entire sides   $\delta$ and $\delta'$  respectively.  But then we can isotope the  union  $\hat \alpha \cup  \hat \alpha'$ across $\Delta$ to get an arc $\alpha''$ that runs     parallel to the third side $\delta''$ of $\Delta$, thus eliminating  the singularity $X \in D$.  Finitely many successive such isotopies   bring $D$ into a position where it is contained in  $\tau$ and has no   singularities with respect to the $I$-fibering:   It is thus carried   by $\tau$.

\end{proof}

\vskip8pt

\begin{remark}\label{lamination1}\rm
Assume that $\Sigma$ is endowed with some hyperbolic structure.   Recall that a {\it (geodesic) lamination}  $\mathcal{L}$  is a  non-empty closed  subset of $\Sigma $ which is a  disjoint union of 
simple geodesics  called {\it leaves}  (see, e.g., Casson-Bleiler  ~\cite{CB}). We say that $\mathcal{L}$ 
is {\it carried} by a train track $\tau$  if each  of its leaves is  carried by $\tau$. An example, which will do for most of  the rest of this paper and doesn't need the hyperbolic structure on $\Sigma$, is simply to take for $\mathcal{L}$ a  (typically rather long) essential simple closed curve, or a finite pairwise  disjoint collection of such.

\end{remark}

\vskip8pt

Given a train track $\tau \subset \Sigma$ which carries a lamination   $\mathcal{L}$ we can obtain 
a new train track as follows:

\vskip8pt

The  train track $\tau$ can be {\it split}  by moving any of the cusp  points $P$  (now called a {\it zipper}), 
which is located on the boundary of a complementary component  $\Delta$ of $\tau$,  in an ``outward" direction with  respect to $\Delta$. In other words, $P$ is moved into the interior of $\tau$. 
The  zipper  $P$ will move along an {\it unzipping path}, which is embedded in the  interior of 
$\tau \smallsetminus \mathcal{L}$. Furthermore the unzipping path is  always  transverse to the 
$I$-fibers.   Two unzipping paths are not allowed to cross each  other.  In case two  zippers meet the same connected component of an  $I$-fiber  in  $\tau\smallsetminus {\mathcal L}$ from different directions, they have to join  up, thus   changing the topology of the train track and of its 
complementary   components.   A situation like this is called a {\it  collision}.  In case of a collision the unzipping  procedure stops. An  unzipping path which covers $\tau$ is called  {\it complete}. 

\vskip8pt

\begin{definition} \label{derivedtt} \rm We say that $\tau$ can be   derived with respect  to 
$\mathcal L$ if we can successively (or  simultaneously, it does not  make any difference) 
unzip every zipper  along a complete unzipping path, without ever  running into a  collision.  
The train  track $\tau'$   obtained by unzipping along paths which  are shortest possible,  complete unzipping paths is  said to be  {\it  derived from $\tau$ with respect to $\mathcal L$}, or simply,  if the 
context is clear, {\it  derived from  $\tau$}.  

\end{definition}

\vskip8pt

\begin{remark}\label{sametype} \rm If the train track $\tau'$ is   derived from a train track $\tau$ then  
every complementary component $\Delta$ of $\tau'$ is of the same type  as the complementary component  of $\tau$ which is contained in  $\Delta$. This follows directly from  Definition \ref{derivedtt},  
since  during the deriving process the unzipping paths never run into collisions. In particular, if $\tau$ is maximal, then so is $\tau'$.

\end{remark}

\vskip8pt

\begin{lemma} \label{carriesimpliesunzip} 
Given a surface $\Sigma$ and  maximal train tracks  $\tau, \tau' \subset  \Sigma$  so that  $\tau'$ is derived from $\tau$, let  $D$ be a simple closed  essential curve carried by $\tau'$.  Then $D$  covers $\tau$.

\end{lemma}

\vskip8pt

\begin{proof} 

Since $D$ is essential, it has to meet some singular $I$-fiber  $\sigma$ of  $\tau'$, by Remark 
\ref{simplyconnected}. Let $P \in \partial \sigma$ be a singular   point: $P$ is also a cusp point on a  complementary component   $\Delta$   of $\tau'$. Let $\delta$ be the side of $\Delta$  adjacent to $P$, on  the side of  $\sigma$. Let $\alpha$ be an arc on $D$ which starts at   $\sigma$ and runs parallel to an initial  segment of  $\delta$.

By the assumption that $\tau'$ is derived from $\tau$, the cusp  point  $P$ is the endpoint  of some unzipping path  $\zeta \subset  \Delta$   which runs parallel in $\tau$ to an initial segment of 
$\delta$,   starting at $P$.  There are two cases:   

\begin{enumerate}

\item[(i)] The path $\alpha$ runs parallel in $\tau$ along all of  the  unzipping  path $\zeta$.

\vskip5pt

\item[(ii)] There is another singular $I$-fiber of $\tau'$ which has  its non-singular endpoint $P'$ on 
$\delta$, while its singular  endpoint is a cusp point $Q$ of another  complementary component 
$\Delta'$ of $\tau'$.  The complementary component $\Delta'$ has two  sides  $\epsilon$  and  $\delta'$ adjacent to $Q$,  such that  the  beginning of $\delta'$  runs parallel to  $\delta$  while $\alpha$  runs along  $\epsilon$.  (i.e., $\alpha$ ``splits''  off $\delta$ at  $Q$.)

\end{enumerate} 

\noindent Now either $\epsilon$  and $\delta'$ (and thus $\delta$)   run parallel on $\tau$ until the 
end of the  unzipping path $\zeta$ is reached, or else the  unzipping  path $\zeta'$  for $\Delta'$ which separates  $\epsilon$  from  $\delta'$ is strictly  shorter than $\zeta$ (where ``length'' is measured in terms of the transverse $I$-fibers that are being  intersected by $\zeta$ and by $\zeta'$). By  this we mean that  there is a proper subpath  of $\zeta$ which is parallel on $\tau$ to  $\zeta'$. We can rule out  this case by assuming at the beginning of  this proof that the arc  $\alpha$ on $D$ starts at a shortest 
possible  unzipping path.   Thus  an initial segment $\epsilon'$ of  $\epsilon$ has to run parallel in 
$\tau$ to some unzipping  path  $\zeta$. If $\alpha$ branches again  off  $\epsilon$ before the end  of 
$\epsilon'$ is reached, we repeat   the exact same argument just given.

This shows  that $\alpha$ has to run parallel to some unzipping  path  for  $\tau'$, and hence it covers 
$\tau$, by the properties of a derived train track as given in    Definition \ref{derivedtt}.

\end{proof}

\vskip8pt

\begin{remark}\label{weaklyderived}\rm   
Since Lemma \ref{carriesimpliesunzip} is the crucial place in this  paper  where we use the properties of a derived train track, it is  natural to  ask whether a weaker condition may still allow the same 
conclusion. Indeed, one  can work with a train track $\tau'$ that is {\it weakly  derived} from  a given train track $\tau$, by which we mean that  $\tau'$ is obtained  by unzipping $\tau$, and that every circuit 
carried by $\tau'$  covers $\tau$.  Here a {\it circuit} means a loop which  embeds under the natural map 
$\tau \to \Gamma_{\tau}$.

As there are only finitely many circuits in any train track, this  definition is still practial to check, and the results of this paper  generalize  to this weaker (but more cumbersome) notion. 

\end{remark}

\vskip8pt

\begin{definition}\label{tower}\rm   
A collection of train tracks  $\tau_{0} \supset \tau_{1} \supset   \ldots \supset\tau_{n}$ will be  
called an  {\it $n$-tower  of derived train tracks in $\Sigma$} if    there is  some lamination  
$\mathcal{L} \subset \Sigma$ so that  $\tau_{i}$ is  derived from   $\tau_{i-1}$ with respect  to  
$\mathcal{L}$ for all $i  = 1, \ldots, n$.

\end{definition}

\vskip8pt

Sequences of nested train tracks, as given in the previous definition,  occur already in \cite{MM}, \S 3.1, where they are used to derive  lower bounds for the distance in the curve complex.  Indeed, the 
following  proposition is a variant of their basic observation, presented here in the slightly more combinatorial setting which we prefer in this paper, for the  benefit of our use in the later sections in the context of Heegaard splittings of 3-manifolds.

\vskip8pt

\begin{proposition}  \label{towerimpliesdistance} 
Let $\tau_{0} \supset \tau_{1} \supset \ldots \supset\tau_{n}$  be  an  $n$-tower of derived  train tracks 
 in $\Sigma$. Assume that  $\tau_{0}$   (and hence of any of the  $\tau_{i}$) is maximal.  Let  $D$ be a simple closed curve  carried  by $\tau_n$.  Then any  essential simple closed curve $D'$   which 
satisfies  $d_{\mathcal{C}}(D, D')\leq n$ can be carried by  $\tau_{0}$. 

\end{proposition}

\vskip8pt

\begin{proof} 
Any curve carried by some of the $\tau_{i+1}$ covers   $\tau_{i}$, by Lemma  \ref {carriesimpliesunzip}. A second curve disjoint from the  first   can be then carried by $\tau_{i}$, by Lemma  \ref{disjointcarried}.  
Hence we see recursively that for any family of  essential simple  closed   curves  
$D = D_{0}, D_{1},  \ldots, D_{n} = D'$, where any subsequent  pair  $D_{i}, D_{i+1}$ is disjoint, each $D_{i}$ can be carried by   $\tau_{n-i}$ and covers  $\tau_{n-i-1}$. In particular, $D'$ can be   carried by $\tau_{0}$.

\end{proof}

\vskip8pt

\begin{remark}\rm
The last proposition can be  reformulated to say that any simple closed   essential curve $E$ that cannot be carried  by  $\tau_0$ satisfies    $$d_{\mathcal{C}}(D, E)\geq n+1 \, .$$

\end{remark}

\vskip8pt

Thus, in order to apply this result in practice, we need a  convenient  criterion to ensure that a given curve  $E$ cannot  be carried by the train track $\tau_0$. For this  purpose  we  introduce, in the next section, the notion of a fat  train track.

\vskip10pt

\vskip5pt

\section{Fat train tracks}\label{fattraintracks}

\vskip4pt

In this section we will investigate a special class of train tracks   that are very useful and natural in the context of this paper.

\vskip10pt

\begin{definition}\label{fatraintrack}\rm 
A train track $\tau \subset \Sigma$ is called {\it fat} if all of its  exceptional fibers are cyclic. We denote by 
${\mathcal E}_{\tau}$ the collection of simple closed curves on   $\Sigma$ given by the exceptional fibers of  $\tau$.

\end{definition}

\vskip5pt

A system $\mathcal E$ of essential simple closed curves on $\Sigma$ is called  a {\it complete decomposing system} if every complementary component  of  $\mathcal E$ in $\Sigma$ is 
a  {\it pair-of-pants}, i.e., a sphere with three open  disks removed. 

If a complete decomposing system is given, then  we will assume that any essential simple closed curves $D$ on $\Sigma$  that is considered  shall be transverse to $\mathcal E$. Similarly, any simple arc $\alpha$ must be  transverse to $\mathcal E$, and we will only consider arcs that have 
their endpoints on $\mathcal E$. As before, such $D$ or $\alpha$ are considered only up to isotopy: 
However, we will only allow isotopies of the pair $(\Sigma, \mathcal E)$. 

A curve $D$ or such an arc $\alpha$  is called {\it tight} with respect to $\mathcal E$  if the 
number of  intersection points with $\mathcal E$ cannot be strictly decreased  by an isotopy of 
$D$ or $\alpha$.

\begin{definition}\label{smallwave}\rm  
Let $P \subset \Sigma$ be a pair-of-pants.

\vskip4pt
\noindent
(a)   A simple arc in $P$ which has its two endpoints on different   components of $\partial P$ will be called a {\it seam}.

\vskip4pt
\noindent
(b)  A simple arc in $P$ which has  both endpoints on the same component of $\partial P$, and is not 
$\partial$-parallel, will be called a {\it wave}.

\vskip4pt
\noindent
(c)   An essential simple closed  curve $D \subset \Sigma$ has a {\it  wave} with respect to a complete decomposing system $\mathcal{E}  \subset \Sigma$ if $D$  is tight with respect to $\mathcal E$ and if it
contains a subarc  that is a wave in a  complementary component $P_{i}$ of  $\mathcal{E}$ in 
$\Sigma$.  

\vskip4pt

\noindent 
(d)   An essential simple closed  curve $D \subset \Sigma$ has a {\it  wave} with respect to a  fat 
train track  $\tau$ if $D$ has a wave with respect to ${\mathcal  E}_{\tau}$, or if $D$ is  isotopic
to some  $E_{k} \in {\mathcal E}_{\tau}$.

\end{definition}

\vskip8pt

Let $P \subset \Sigma$ be a pair-of-pants. Place a {\it vertex} on  each of  its three boundary components and connect any two of them by  disjoint  simple arcs.  The  three arcs together bound  a subsurface $\Delta_1 \subset P$, called  a {\it triangle}. Repeat this operation on 
$P  \smallsetminus  \Delta_1$  to obtain a second such triangle denoted  by $\Delta_2$. 
Note that $P \smallsetminus \int(\Delta_1 \cup \Delta_2)$ is a  collection of  three {\it rectangles} 
$R_1, R_2$ and  $R_3$, since the boundary of each of $R_j$ (for  $j =  1,2,3$) is  composed of four 
arcs separated by four  of the six vertices introduced above.  Define  an $I$-fibration on  each of the rectangles $R_j$  by filling $R_{j}$  with arcs  parallel to the two arcs from 
$\partial R_j \cap \partial P$.

\vskip5pt

Let $\mathcal{E}$ be a {complete decomposing system} on  $\Sigma$, and consider the  collection of complementary pair-of-pants  $P_{i} \subset \Sigma$, i.e., $\Sigma = \bigcup_{ i = 1}^{2g - 2} P_i$. 

Put the structure as above on each of pair-of-pants $P_i$,  in such a way that on  each curve $E_k,$  one has placed four distinct vertices:  two vertices  from each of the two pair-of-pants adjacent to 
$E_{k}$.

The $I$-fibrations on any of the rectangles in each of the $P_i$ join  up to define a train track structure 
$\tau$ on $\Sigma  \smallsetminus \int(\Delta)$, where $\Delta$ is the  union of all the  triangles 
$\Delta_j$ in any of the $P_{i}$. For this train track the set of  exceptional  fibers ${\mathcal E}_{\tau}$ 
is exactly the complete decomposing system $\mathcal E$ we started out  with, and they are all cyclic, 
i.e., $\tau$ is a fat train track. Furthermore, the train track   $\tau$ has no complementary region other than  the two triangles in  each of the pair-of-pants $P_{i}$, and those are indeed triangles in 
the meaning of Section \ref{crvcmplx}. Finally, consider any arc $\alpha$ that intersects $\mathcal E$ precisely in its  endpoints: we observe that $\alpha$ can be carried by  $\tau$ if and only if it is a seam: In particular, no wave (with  respect to ${\mathcal E}_{\tau}$) is carried by the train track  $\tau$.
 
This construction gives rise to the following:

\vskip8pt

\begin{definition} \label{completefatraintrack}\rm 
A fat train track $\tau \subset  \Sigma$ is called {\it complete}  if  the following conditions are satisfied:
\begin{enumerate}

\item The collection $\mathcal{E}_{\tau}$ of exceptional fibers of  $\tau$ is a complete decomposing system on  $\Sigma$.

\item Each pair-of-pants $P_{i}$ complementary to the system  $\mathcal{E}_{\tau}$ contains two triangles as complementary  components of $\tau$ in $P_{i}$. 

\item 
The train track $\tau$ only carries seams, but no waves, with respect  to the complete decomposing system $\mathcal{E}_{\tau}$.

\end{enumerate}

\end{definition}

\vskip8pt

\begin{remark}\label{finitelymany} \rm 
Given a complete decomposing system $\mathcal{E} \subset \Sigma$,   there are only  finitely many complete fat train tracks $\tau$ with    $\mathcal{E}_{\tau} = \mathcal{E}$, up to   orientation preserving
homeomorphisms of  $\Sigma$ which leave invariant every curve of $\mathcal{E}$ and every  complementary pair-of-pants.

To be precise, for each curve $E \in \mathcal{E}$ there are six  possible configurations  for the vertices of the adjacent four  triangles, up to a  homeomorphism as above.  Hence we have  $6^{3g - 3}$ possible configurations for a complete  fat train track, up to such  homeomorphisms.

By definition every complete fat train tack is maximal.

\end{remark}

\vskip8pt

\begin{lemma} \label{fatcarries}
Let  $\mathcal{E} \subset \Sigma$ be a  complete decomposing system.  Any essential simple closed curve $D \subset \Sigma$ which does not have  a wave with respect to   $\mathcal{E}$, and is not parallel to  any $E_{k} \in \mathcal{E}$, is carried by some complete fat train  track $\tau$ with exceptional fibers  ${\mathcal E}_{\tau} = {\mathcal E}$.

The same is true for any system of pairwise disjoint essential simple  closed curves which satisfy the same conditions as $D$.

\end{lemma}

\vskip4pt

\begin{proof}
 If $D$ is not parallel to any $E_{k} \in \mathcal{E}$    and has no waves with respect  to $\mathcal{E}$,  then for each of the pair-of-pants $P_{i}$   complementary to  $\mathcal{E}$ in  $\Sigma$ the connected components of $D \cap P$  must all be seams.   Since any two seems which  join the same two boundary components of  $P_{i}$ are necessarily  parallel (because of the simple topology of  a pair-of-pants), it  follows that the parallelity classes of such seems can be grouped  together and  isotoped so that they are contained in a transversely  $I$-fibered rectangle in $P_{i}$  as  introduced above. After introducing, if necessary, some  additional  empty $I$-fibered rectangles,  the  complement of the rectangles in each $P_{i}$ will consist  precisely of two triangles. Thus the  union  of all   $I$-fibered rectangles defines a complete fat train  track  that carries $D$ and satisfies 
${\mathcal  E}_{\tau} = {\mathcal E}$.

\end{proof}

\vskip0pt

The same type of argument can be used to obtain the following statement:

\vskip0pt

\begin{remark} \label{fillingpants} \rm
Let $\mathcal E$ be a complete decomposing system on the surface  $\Sigma$, and let $D$ be an essential simple closed curve (or a system of  such curves) on $\Sigma$ that is tight with respect to 
$\mathcal E$. We say that $D$ {\em fills}  a pair-of-pants $P$ complementary to $\mathcal E$, if 
$D \cap P$  is the disjoint union of  precisely  3  distinct isotopy classes of intersection arcs.
As before we consider here isotopy of the pair $(P, \partial P)$. Then the following three statements are equivalent:

\begin{enumerate}
\item[(1)]
The curve $D$ fills every pair-of-pants complementary to $\mathcal E$, and none of the intersection arcs is a wave.

\item[(2)]
There exists a unique complete fat train track $\tau$ with exceptional fibers 
$\mathcal{E}_{\tau} = \mathcal{E}$ that carries $D$.

\item[(3)]
There exists some complete fat train track $\tau$ with exceptional fibers 
$\mathcal{E}_{\tau} = \mathcal{E}$ that is covered  by $D$.
\end{enumerate}

\end{remark}

\vskip0pt

\begin{lemma}\label{distancethree}  Let $\tau$ be a complete fat train track, and let $\tau'$ be derived  
from $\tau$. Let $D$ be an essential simple closed curve that is carried by $\tau'$,  and let $E$ be an  
essential simple closed curve that has  a wave with respect to $\tau$.   Then 
$$d_{\mathcal{C}}(D, E) \geq 2  $$

\end{lemma}

\vskip0pt

\begin{proof} 
We apply Lemma \ref{carriesimpliesunzip} to deduce that the curve  $D$  covers $\tau$. But any curve 
that covers $\tau$ must intersect any of the $E_{j} \in {\mathcal   E}_{\tau}$, and also any wave in any of 
the complementary components $P_{i}$ of  the complete decomposing   system ${\mathcal E}_{\tau}$ in 
$\Sigma$. Thus $D$ intersects the curve $E$ and hence it is of   distance at least two from it:
$d_{\mathcal{C}}(D, E) \geq 2$ 

\end{proof}

\vskip8pt

\begin{corollary}  \label{corlargedistance}
Let $\tau_{0} \supset \tau_{1} \supset \ldots  \supset\tau_{n}$, $n  \geq 1$,  be an $n$-tower  of derived train tracks in  $\Sigma$.  Assume that $\tau_0$ is a complete fat  train track.  Let $D$ be an  essential simple closed curve  carried by $\tau_n$, and let $E$ be an  essential simple closed   curve which has a wave with respect to  $\tau_0$. Then one has:
$$d_{\mathcal{C}}(D, E)\geq n+1$$

\end{corollary}

\vskip8pt

\begin{proof} 
By Proposition \ref{towerimpliesdistance} any  curve $D'$  of distance  at most $n -1$ from $D$ is carried by $\tau_{1}$. Hence we can apply Lemma \ref{distancethree} to deduce that $D'$ has distance 
greater or equal to $2$ from $E$.  Thus  $d_{\mathcal{C}}(E, D) \geq  n + 1$. 

\end{proof}

\vskip8pt

We now prove a useful analogue of Lemma \ref{disjointcarried}, for  simple arcs rather than simple closed curves.   In order to ``fix'' the endpoints  of such an arc $\alpha$ we require here and in the next 
sections  that $\partial \alpha$ is contained in a complete decomposing system $\mathcal E$, and that 
$\alpha$ is  tight with respect to $\mathcal E$. The system $\mathcal E$, in our  context, is given as set of  singular fibers $\mathcal{E} = \mathcal{E}_{\tau}$ of some complete  fat train track $\tau$. Thus the role of this train track $\tau$ is in some  sense  that of a ``coordinate system'', while the train track 
$\tau'$,  the  one we are really  interested in, is a much finer and longer train track than $\tau$:  we only require that $\tau'$ is derived from $\tau$.

\begin{lemma}\label{disjointcarried2}  
Let $\tau$ be a  complete fat train track on  a  surface $\Sigma$, and let $\tau'$ be  a maximal train 
track  derived from $\tau$.  Let   $\beta$ be an arc with endpoints on $\mathcal{E}_{\tau}$ which covers $\tau'$. Let  $D$ be an  essential simple closed curve which is tight with respect to  $\mathcal{E}_{\tau}$ and contains $\beta$ as subarc. Then  $D$ can be  carried by $\tau'$, and in fact covers $\tau'$.

\end{lemma}

\begin{proof}
We first observe that the analogue  of the statement of Lemma \ref{entireside} holds, with $D$ replaced 
by $\beta$: the proof given in Section 2 applies word by word to this generalization.

We now proceed precisely as in the proof of Lemma  \ref{disjointcarried}: We first use an isotopy that fixes $\beta$ and  all intersection points of $D$ with  $\mathcal{E}_{\tau}$ to isotope $D$ off 
all complementary components of $\tau'$,  while making sure that it  stays tight with respect to 
$\mathcal{E}_{\tau}$,  and also that it stays  simple.  We then decrease the number of singularities successively,  as in the  proof of Lemma  \ref{disjointcarried} (using the above mentioned analogue  of Lemma  \ref{entireside}). At the end state of this procedure there are no singularities left, and $D$ has been isotoped   into a position where it is carried by $\tau'$. But $D$ contains $\beta$ as subarc, which covers $\tau'$ by  assumption. Thus $D$ covers $\tau'$.

\end{proof}

\vskip5pt

Let $k$ be an essential simple closed curve which is tight with  respect to the complete decomposing system $\mathcal E$ on $\Sigma$. The number  $|k \cap \mathcal{E}|$ of  intersection  points of  $k$ with $\mathcal E$ is called the  {\it $\mathcal E$-length}  of $k$ and is denoted by  
$\mid k \mid_{\mathcal E}$.  The same definition and notation will be  used for a simple arc $\alpha$ instead of $k$, where as before we require $\partial \alpha \in \mathcal E$.

Recall that two tight simple arcs  $\alpha, \alpha'$ on $\Sigma$ are called {parallel} (with respect to 
$\mathcal E$) if, after orienting them properly, they intersect $\mathcal E$ in precisely the same sequence of curves  $E_{i_1},\dots, E_{i_k}$, and if the intersections occur from the same direction.  This is  equivalent to saying that the arcs are isotopic by an isotopy of the  pair $(\Sigma, \mathcal{E})$.

\vskip5pt

Let $c$ and $k$ be distinct simple closed curves on $\Sigma$ that   are tight with respect to 
$\mathcal  E$, and let $P \in c \cap k$ be an intersection point.  We denote by 
$\mid P \mid_{\mathcal E}$ the $\mathcal E$-length of  any of two maximal  arcs on $k$ and  on $c$, 
which are  parallel and  which both contain $P$.  In the last sentence, the terminology ``arc on a closed curve'' needs to be  specified: Such an arc is not necessarily a subarc, but it can also be  an arc that winds several times around the closed curve, thus being  immersed but not embedded in the curve.

We define the {\it   twisting number of $k$ along $c$ at $P$} to be the quotient:  
$$tw_{P}(k,c) = \frac{\mid P \mid_{\mathcal E}}{\mid c  \mid_{\mathcal  E}}$$

\vskip5pt

\begin{lemma}
\label{twistingnumber}

Let $\mathcal{E}$ be a complete decomposing system on $\Sigma$, and  let $k$ and $c$ be essential simple curves on $\Sigma$ that are tight  with respect to $\mathcal E$.  Assume that $k$ and $c$ only intersect  essentially, and let $P \in c \cap k$ be such an essential  intersection point.

\vskip4pt

\begin{enumerate}
\item[(a)] The Dehn twist $\delta_{c}$ at $c$ effects the 
twisting number as follows:
$$ tw_{P}(\delta^m_{c}(k), c) = |tw_{P}(k, c)  + \varepsilon m |\, ,$$
where $\varepsilon = \pm 1$ is independent of $m$.

 \vskip4pt

 \item [(b)] Assume  $\mid tw_{P}(k, c) \mid \, \, > 1$ and assume   furthermore that $c$ covers 
some train track $\tau'$ that is derived from a complete fat train track  $\tau$ with  
$\mathcal{E}_{\tau} =  \mathcal{E}$.  Then $k$ covers $\tau'$ as well. In particular, $k$ does not have waves  with respect to $\mathcal E$.
\end{enumerate}

\end{lemma}

\begin{proof}

Statement (a) of this lemma is a direct consequence of the above  definitions.   Here the role of the constant $\varepsilon$ is geometrically  explained as follows:  The Dehn twist $\delta_{k}$ is a 
``right-handed'' twist. On the other hand, before applying the Dehn  twist, the curve $k$ may already, at the intersection point $P$, wind around $c$, either in the ``right  hand'' sense, or in the ``left hand'' one. In the first case one has  to set $\varepsilon = 1$, while in the second case one has $\varepsilon = -1$.

To prove (b), we  first note that $k$ is not assumed  to be carried by $\tau'$. However, it suffices to apply Lemma  \ref{disjointcarried2}, where $\beta$ is the maximal subarc on $k$ which  contains $P$ and is parallel to an arc on $c$ that also contains $P$.

\end{proof}

\vskip8pt

\begin{remark} \label{laminationversuscurve} \rm
(a)  Since in Lemma   \ref{twistingnumber} $k$ and $c$ are both simple, it follows directly 
from the  definitions that  for any two intersections points $P, Q \in c \cap k$  the twisting numbers 
$tw_{P}(k, c)$ and $tw_{Q}(k, c)$ cannot differ  by more than 1.  The same is true for $tw_{Q}(k', c)$ rather than  $tw_{Q}(k, c)$, where $k'$ is any second essential simple closed curve  disjoint from $k$ that is tight with respect to $\mathcal E$ and  intersects $c$ essentially in $Q$.

\smallskip
\noindent (b) 
The above definitions as well as the statement of Lemma   \ref{twistingnumber} stay valid if the curve 
$k$ is replaced by a   lamination $\mathcal L \subset \Sigma$,  where every leaf of $\mathcal  L$ is supposed to be tight with respect to $\mathcal E$.

\end{remark}

\vskip4pt

\begin{proposition} \label{surgerylemma}
Let $\mathcal{D}$ and $\mathcal{E}$ be  complete decomposing systems on $\Sigma$, and
let $k$ be an essential simple curve on $\Sigma$ that is tight  with respect to both, $\mathcal{D}$ and 
$\mathcal{E}$.  Assume that  the curve $k$ fills every pair-of-pants complementary to $\mathcal D$ 
or to $\mathcal E$, and that, furthermore, $k$ contains  no wave with respect to either  $\mathcal D$ or 
$\mathcal E$. 

Then  there exists an integer $m_{0}$ such that for  every  
$m \in \mathbb{Z} \smallsetminus \{m_{0}, m_{0}+1, m_{0} +   2, m_{0}+3\}$  one has:

\begin{enumerate}
\item[(a)]

The complete decomposing system  $\mathcal{D}^m = \delta_{k}^m(\mathcal{D})$, obtained from 
$\mathcal D$  via $m$-fold  Dehn twist on $k$, has the property  that $\mathcal{D}^m$ and 
$\mathcal{E}$ do not have  waves with respect to each other. 

\item[(b)]
If $k$ covers a maximal train track $\tau'$ that is derived from  some complete fat train track  $\tau$
with  exceptional fibers $\mathcal{E}_{\tau} = \mathcal E$, then $\mathcal{D}^m$  as well covers 
$\tau'$  (and hence is carried by $\tau'$).

\end{enumerate}
\end{proposition}

\vskip5pt

\begin{proof}
(I.) In this first part of the proof we only consider the system  $\mathcal D$, and we want to investigate which values $m \in  \mathbb{Z}$ have the property (actually a slightly stronger one) that
the system $\mathcal{D}^m$ does not contain waves with respect  to $\mathcal E$.

\smallskip

We first consider an essential intersection point $P$ of $k$ with  some curve $D_{i}$ of the system 
$\mathcal D$. We apply statement (a)  of Lemma \ref{twistingnumber} to see that one of the following two cases occurs, according to whether $tw_{P}(D_{i}, k)$ is an integer (case (1)) or not  (case (2)): 

\begin{enumerate}
\item [(1)]
There exists an integer $m_{1} \in \integers$ such that 
$$tw_{P}(\delta^{m}_{k}(D_{i}), k) = |m - m_{1}|$$
for all $m \in \mathbb{Z}$.

\item [(2)]
There exists an integer $m_{1} \in \integers$  and real numbers $c_{-}, c_{+} \in \, \, (0, 1)$
such that:
\begin{itemize}
\item[(a)]
    $tw_{P}(\delta^{m_{1}}_{k}(D_{i}), k) = c_{-}$
\item[(b)]
$tw_{P}(\delta^{m_{1}+1}_{k}(D_{i}), k)  = c_{+}$

\item[(c)]
$tw_{P}(\delta^{m}_{k}(D_{i}), k) = |m - m_{1} -1 + c_{+}|$ if $m \geq m_{1} +1$

\item[(d)]
$tw_{P}(\delta^{m}_{k}(D_{i}), k) = |-m + m_{1} + c_{-}|$ if $m \leq  m_{1}$.

\end{itemize}
\end{enumerate}

\noindent
A second curve $D_{j} \in \mathcal D$ and any intersection point $Q \in  D_{j} \cap k$ satisfies the same statement, and furthermore we observed  in Remark \ref{laminationversuscurve} (a) that 
$tw_{P}(\delta^{m}_{k}(D_{i}), k)$ and $tw_{Q}(\delta^{m}_{k}(D_{j}), k)$ can  not differ by more than 1. 

\smallskip

We now want to define the  {\it exceptional $\mathcal D$-values} $m_{-}\, , \ldots,  m_{+}$ in 
$\mathbb{Z}$ by defining their complement in $\mathbb{Z}$:

An integer $m$ is a {\it non-exceptional} $\mathcal D$-value if at every intersection point $P$ of 
$\delta^{m}_{k}(D_{i})$ with $k$,  for any $D_{i} \in \mathcal D$, one has 
$$tw_{P}(\delta^{m}_{k}(D_{i}), k) \geq 1 \, ,$$
and for at least one such $P$ and $D_{i}$ the inequality is  strict.

By checking for each $D_{i} \in \mathcal D$ the two possible cases (1) and (2) above, 
we deduce that there are  at most three adjacent {\it exceptional $\mathcal D$-values} 
$m_{-}\, , \ldots,  m_{+}$ in  $\mathbb{Z}$. Furthermore, we deduce from the above definition of the 
non-exceptional values, that for every integer $m \geq m_{+} + 1$ the sum $S_{m}$ 
of all twisting numbers, over all intersection points of any 
$\delta^m_{k}(D_{i})  \in \delta^m_{k}(\mathcal D)$ with $k$, satisfies:
$$S_{m} = S_{m-1} + |\mathcal D \cap k| $$
Similarly, for any integer $m \leq m_{-} - 1$ we obtain:
$$S_{m} = S_{m+1} + |\mathcal D \cap k|$$
We finish this first part of the proof with the observation that the sum $S_{m}$, for any
$m \in \mathbb{Z}$, can also be expressed as quotient
$$S_{m} = \frac{|\mathcal{D}^m \cap \mathcal{E}|- c_{\mathcal E} }{|k|_{\mathcal E}}  \, ,$$
where $c_{\mathcal E}$ is a constant independent of $m$: It is 
equal to the number of intersection points of $\mathcal D$ with $\mathcal E$ that  are not 
contained in any of the arcs on some $D_{i} \in \mathcal D$  which are parallel to an arc on $k$, where the two arcs intersect  essentially, after possibly an isotopy of $D_{i}$ or $k$.

\medskip

Before starting with the second part of the proof, i.e. the  comparison between the exceptional 
$\mathcal D$-values and the  exceptional $\mathcal E$-values, let us observe the following 
properties of the exceptional values:

\begin{enumerate}
\item[(a)]
For every non-exceptional $\mathcal D$-value $m \in  \mathbb{Z}$ there  is at least one of the curves 
$D_{i} \in \mathcal D$ which  intersects $k$ essentially in some point $P$ and satisfies 
$tw_{P}(\delta^m_{k}(D_{i}, k)) > 1$.  Thus $D_{i}$ is carried by the  complete fat train track $\tau$ with exceptional fibers  $\mathcal{E}_{\tau} = \mathcal E$ that is defined by $k$ (compare 
Remark \ref{fillingpants}), and in particular $D_{i}$ has no wave with  respect to $\mathcal E$.  
Since $\tau$ is maximal and the other  $D_{j} \in \mathcal D$ are disjoint from $D_{i}$, the previous 
statement is true for all of $\mathcal D$.

\item[(b)]
We stated above that very integer $m \geq m_{+} +1$ satisfies $S_{m} = S_{m-1} + |\mathcal D \cap k|$.
However, it is possible that the value $m = m_{+}$ also satisfies this equality: This occurs if
at every intersection point $P$ of $\delta^{m_{+}}_{k}(D_{i})$ with $k$,  for any $D_{i} \in \mathcal D$, one has:  $$tw_{P}(\delta^{m_{+}}_{k}(D_{i}), k) = 1$$
The analogous observation is true for $m_{-}$ rather than $m_{+}$.

We conclude that there are two linear (or, rather, ``affine'')  functions $m \mapsto S_{+}(m)$ 
and $m \mapsto S_{-}(m)$ such that  $S_{m + \varepsilon} = S_{+}(m + \varepsilon)$ or 
$S_{m + \varepsilon} = S_{-}(m + \varepsilon)$ for all $\varepsilon \in \{-1, 0, 1\}$ and
$m \in  \mathbb{Z} - N_{\mathcal D}$, where $N_{\mathcal D}$ is a subset of the $\mathcal D$-exceptional values $m_{-}, \ldots,  m_{+}$, and $N_{\mathcal D}$ has either  cardinality 3, 2, or 1. Furthermore, if $N_{\mathcal D}$ has  cardinality 1 and there are precisely three exceptional $\mathcal 
D$-values, then it is the middle one of those that is contained in $N_{\mathcal D}$.

\end{enumerate}

\medskip

\noindent (II.) By considering $\mathcal{E}^m =  \delta_{k}^{-m}(\mathcal{E})$ and 
the pairs $\mathcal{D},  \mathcal{E}^m$ we observe that the hypothesis in the 
Proposition is  symmetric in  $\mathcal D$ and  $\mathcal E$, and that the same arguments as in part (I.) apply to $\mathcal  E$ rather than $\mathcal D$.  We now note that  
$|\mathcal{D}^m \cap \mathcal{E}| = |\mathcal{D} \cap \mathcal{E}^m|$,  since the homeomorphism 
$\delta^{-m}$ preserves the intersection  number.  Furthermore, from the last equation in part (I.) before observation  (a), we obtain the expression
$$|\mathcal{D}^m \cap \mathcal{E}|= |k|_{\mathcal E} \, S_{m} + c_{\mathcal E}\, ,$$
where $|k|_{\mathcal E}$ and  $c_{\mathcal E}$ are constants independent of $m$. This allows us to 
conclude that the ``non-linear" exceptional values must satisfy:
$$N_{\mathcal D} = N_{\mathcal E}$$
Thus we can deduce from observation (b) above that the union of the$\mathcal D$-exceptional and the 
$\mathcal E$-exceptional values  consists of maximally 4 elements.  Thus observation (a) above finishes 
the proof.

\end{proof}

\vskip10pt

\begin{remark} \label{twistingdistance2} \rm
 From a careful analysis of the details of the proof of Proposition  \ref{surgerylemma} it seems possible that, under the additional assumption that  at  some intersection point the  $\mathcal D$-length 
of the  intersection arc is not an integer, and similarly  for the $\mathcal E$-length, in the statement of Proposition \ref{surgerylemma} one does not have to exclude  four but only three or even fewer exceptional Dehn twist exponents.   Indeed, statement (a) of that proof may in fact be true also for some of the  exceptional $\mathcal D$-values $m_{-}, \ldots, m_{+}$. 

\end{remark}

\vskip5pt

\section{
Heegaard splittings}\label{applications}

\vskip4pt

Let $H$ be a 3-dimensional handlebody of genus $g \geq 2$, and let $\Sigma = \partial H$ denote its boundary surface.  The set  $\mathcal{D}(H)$ of isotopy classes of  essential simple closed curves 
on $\Sigma$ that bound a disk in $H$ is a subset of $\mathcal{C}^0(\Sigma)$. It is the vertex set 
of what is called the {\it disk complex} of the handlebody $H$,  contained as a subcomplex in
$\mathcal{C}(\Sigma)$.  

\vskip10pt

Similarly, we consider complete decomposing systems, up to isotopy in 
$\Sigma$, which bound disk systems in $H$, and denote the set of such 
isotopy classes by $\mathcal{CDS}(H)$.  As in previous 
sections, we will sometimes omit the distinction between systems of  
curves and their isotopy classes, 
to make the notation easier.

\vskip10pt

The most prominent place where 3-dimensional handlebodies appear in   topology are 
{\it Heegaard splittings} of $3$-manifolds: Let $M$ be a closed   orientable $3$-manifold, 
and let  $\Sigma \subset M$ be a {\it Heegaard surface} of genus $g   \geq 2$. This means 
that $M$ decomposes along $\Sigma$ into two genus $g$  handlebodies  $V$ and $W$, 
so that $M = V \cup_{\Sigma} W$.  

\vskip10pt

The {\it  distance} of a  Heegaard splitting  $M = V \cup_{\Sigma} W$  is defined by
$$d(V, W)   = \min \{d_{\mathcal C}(D, E) \mid D\in \mathcal  {D}(V),  E \in \mathcal {D}(W) \}\, ,$$ 
where $d_{\mathcal C}$ denotes, as before, the distance in the curve   complex $\mathcal{C}(\Sigma)$ (see \cite{He}).

\vskip10pt

\begin{remark} \label{aside}\rm 
Assume that $M$ is irreducible, i.e., every embedded essential $2$-sphere bounds a $3$-ball in $M$.
If  $M$ has a Heegaard splitting  $M = V  \cup_{\Sigma} W$  with  distance $d_{\mathcal C}(V, W) = 0$,  
then the splitting is called {\it  stabilized}. A Heegaard splitting  which satisfies  $d(V, W) \leq 1$ is called {\it weakly reducible}.  Heegaard  splittings with $d(V, W) \geq 2$ have been termed by  Casson-Gordon 
(see ~\cite{CG})  {\it strongly irreducible}. They play an important   role in  $3$-manifold theory.

\end{remark}

\vskip10pt

A well known  and easy  observation states  the following:

\begin{remark} \label{otherdisks} \rm
Given a complete decomposing system 
$${\mathcal D} = \{D_1, ..., D_{3g-3}\} \in \mathcal{CDS}(V)$$ 
for a  handlebody $V$, then any other essential disk-bounding curve  $D\in  \mathcal{D}(V)$ is either parallel to one of  $D_i $, or $D$  has a wave with respect to $\mathcal D\,$ (i.e., $D$ contains a wave  
in one of the pair-of-pants  $P_{j} \subset \Sigma$ complementary to $\mathcal D$, compare 
Definition \ref{smallwave} (c)). 

\end{remark}

One should keep in mind, however, that a curve $D$ on  $\Sigma = \partial  V$ may well contain a wave with  respect to some  ${\mathcal D} \in \mathcal{CDS}(V)$, even if $D$  does not bound a disk  in $V$.

\vskip10pt

A  complete decomposing system  ${\mathcal D} = \{D_1, ...,  D_{3g-3}\}  \subset \Sigma$ is said to  
{\it have a wave with respect to a  second  complete decomposing  system 
${\mathcal E} \subset \Sigma$}  if some of the $D_{i}$ has a wave  with respect to  $\mathcal E$.

\vskip10pt

\begin{lemma}[\cite{He}, Lemma 1.3] \label{quotehempel}  
For every Heegaard splitting of a 3-manifold $M = V \cup_{\Sigma} W$   there always  exists a pair of complete decomposing systems  ${\mathcal D} \in  \mathcal{CDS}(V)$ and 
${\mathcal E \in \mathcal{CDS}(W)}$ which have no waves with respect   to each other.

\end{lemma}

\vskip4pt

Consider a complete decomposing system $\mathcal{D} \subset\Sigma$,   and an essential curve 
$E \subset \Sigma$.   There always is an isotopy of $E$ which   makes $E$ {tight} with respect to 
$\mathcal D$, i.e., it eliminates all inessential intersection points  from $E \cap \mathcal{D}$,  
so that $E$ is   cut by  $ \mathcal{D}$ into arcs which are either seams or waves  (see 
Definition \ref{smallwave} and the preceeding discussion). As in  Section 3, we will  consider these arcs only up to {\it isotopy}, by which we  mean  isotopy of the pair $(\Sigma, \mathcal{D})$.

\vskip4pt

Let $\tau \subset \Sigma$ be a maximal train track, i.e., all connected components of 
$\Sigma \smallsetminus \tau$ are  triangles.   Let  $\mathcal{D}$  be a complete decomposing  
system in $\Sigma$ which is  carried by $\tau$.  Then every   connected component $P$ (always a pair-of-pants !) of  $\Sigma \smallsetminus \mathcal D$  contains precisely two of the  triangles 
complementary  to $\tau$.  We consider the following two possibilities:

\begin{enumerate}
\item Every wave in $P$ can be carried by $\tau$, while every  seam can be isotoped  into some of the $I$-fibers of $\tau$. We say  that in this case $P$ has {\it $\Theta$-graph} shape.

\vskip4pt

\item One of the  seams and two non-isotopic  waves can be carried by  $\tau$,  while a third 
wave as well as other two  non-isotopic seams can be isotoped into  some of the $I$-fibers of $\tau$. 
In this case  $P$ is said to have {\it eye glass} shape.

\end{enumerate}

\vskip5pt

\begin{lemma} \label{cdsontt}
Let $\tau \subset \Sigma$ be a maximal train track.   Let   $\mathcal{D}$ be a complete decomposing  system in $\Sigma$ which is   carried by $\tau$.   Then every  pair-of-pants $P$  complementary to 
$\mathcal D$ in $\Sigma$ has either 

\begin{enumerate}
\item[(1)] $\Theta$-graph shape, or 

\item[(2)] eye glasses shape.

\end{enumerate}

\end{lemma}

\vskip4pt

 \begin{proof} 
Since $\mathcal{D}$ is carried by $\tau$, the train track $\tau$   defines an induced train track 
$\tau_{P} = \tau \cap P$  on each pair-of-pants $P$, such that the $I$-bundle structure 
on $\tau$ restricts to the $I$-bundle structure on $\tau_{P}$. It follows from an easy Euler characteristic count that each   pair-of-pants complementary to $\mathcal D$ must contain precisely two   of the triangles that are complementary to $\tau$.

Consider now the deformation of the induced train track  $\tau_P$ onto the graph 
$\Gamma_{\tau_P}$ as defined in Remark \ref{traintrack}. There are two possible cases.  The first,   
corresponding to case (1), is that  $\Gamma_{\tau_P}$ is a $\Theta$-graph.  The second, corresponding to   case (2), is that  $\Gamma_{\tau_P}$  is composed of two circles connected by an arc:  
it is an ``eye glasses" graph.  

Recall that a wave in $P$ is an arc from  one  boundary component to  itself which is not 
$\partial$-parallel, and that a seam in $P$ is an  arc  connecting two different boundary components of $P$.   It is   easy to  check  that in case (1) all waves are carried by $\tau_P$ and  all  seams are isotopic into the $I$-fibers, so that $P$ has   $\Theta$-graph shape. In case (2) one of the  seams and two 
non-isotopic  waves can be carried by $\tau_P$, while a  third  wave as well as  other two non-isotopic  seams can be isotoped  into some of the $I$-fibers of $\tau_P$: the  pair-of-pants $P$ has an eye glasses shape.

\end{proof}

\vskip4pt

Notice that, contrary perhaps to the impression given above, there is more than one pair-of-pants,
up to homeomorphisms that respect the complementary components and the singular $I$-fibration,
which has eye glasses shape, and more  than one which has $\Theta$-graph shape.  Indeed, 
the homeomorphism type of  such pair-of-pants depends also on the direction of the train track switches at the singular fibers.

\vskip4pt

\begin{remark} \label{secondcase}\rm 
Note that case (2) of Lemma \ref{cdsontt} cannot occur for a   complete fat train track $\tau$ which 
has $\mathcal E$ as exceptional fibers, if $\mathcal E$ does not  have  waves with respect to 
$\mathcal D\,$: Indeed, waves which can be isotoped into the   $I$-fibers of $\tau$ are parallel to arcs on $\mathcal E$ which then would also be waves with respect to   $\mathcal{D}$.

Conversely, if each of the  pair-of-pants complementary to $\mathcal  D$   has
$\Theta$-graph shape, then $\mathcal E$ does not have waves with respect  to $\mathcal D$.

\end{remark}

\vskip4pt

\begin{lemma} \label{carriedwaves}
Let $\tau \subset \Sigma$ be a  complete fat  train track, and let $\tau'  \subset \tau$ be a train 
track derived from $\tau$.  Let $\mathcal{D}$ be a complete decomposing system in $\Sigma$ 
which is carried by $\tau'$,  with the property that every  pair-of-pants complementary to 
$\mathcal  D$ has $\Theta$-graph shape. Let  $D \subset \Sigma$  be an essential simple 
closed curve which is tight with  respect to $\mathcal D$, and assume  that some arc $\beta$ from 
the set of  arcs $D - \mathcal D$  is a wave with respect to  $\mathcal D$. Then $D$  can be carried by 
$\tau$.

\end{lemma}

\vskip10pt

\begin{proof} 
Let $P$ be the pair-of-pants complementary to $\mathcal D$ that   contains the wave $\beta$. By assumption $P$ has $\Theta$-graph  shape,  so that $\beta$ is carried by $\tau'$. We observe that any 
wave in such a $\Theta$-graph shaped pair-of-pants $P$ has to run parallel  on $\tau'$ to at least one entire side of one  of the two  connected  components complementary to $\tau'$ which are contained in 
$P$. But this implies immediately that $\beta$  contains a subarc $\beta'$ that has to run parallel to 
some complete unzipping path which is used to derive $\tau'$ from   $\tau$, and thus $\beta$ covers 
$\tau$.  Note that $\beta'$ has its endpoints on  $\mathcal{E}_{\tau}$. The rest of the proof is now a direct application of Lemma  \ref{disjointcarried2}, applied to the subarc of $\beta'$ of $\beta$.

\end{proof}

\vskip10pt

\begin{theorem} \label{HSdistance}
Let $M$ be an oriented $3$-manifold with a Heegaard splitting $M = V   \cup_{\Sigma} W$. Consider 
complete decomposing systems $\mathcal D \in \mathcal{CDS}(V)$ and  
$\mathcal E \in \mathcal{CDS}(W)$ which do not have waves with   respect to each other.   Let 
$\tau \subset \Sigma$ be a complete fat train track with   exceptional fibers 
$\mathcal{E}_{\tau} = \mathcal E$, and assume that $\mathcal D$ is   carried by $\tau_{n}$, for some
$n$-tower of derived train tracks  $\tau = \tau_{0} \supset \tau_{1}  \supset \ldots \supset \tau_{n}$
with $n \geq 2$.  Then the distance of the  given Heegaard splitting  satisfies: $$d(V, W) \geq n$$

\end{theorem}

\vskip10pt

\begin{proof}
By hypothesis the system $\mathcal D$ is carried by $\tau_{n}$.  Let  $D \in \mathcal{D}(V)$ be 
any  disk-bounding essential simple closed curve in $V$.  From  Remark \ref{otherdisks} we know 
that either $D \in \mathcal{D}$, or else $D$ contains a wave with  respect to $\mathcal D$.  From Lemma \ref{cdsontt} and Remark  \ref{secondcase} we deduce that this wave is carried by $\tau_{n}$.  
Thus we can apply Lemma \ref{carriedwaves}, and obtain that $D$ is  carried by $\tau_{n-1}$.

On the other hand, any essential disk-bounding simple closed curve $E  \in \mathcal{D}(W)$ in $W$ has a wave with respect to $\tau =  \tau_{0}$, by Definition \ref{smallwave} (d) and Remark  \ref{otherdisks}.  Thus Corollary \ref{corlargedistance} gives the  desired inequality.

\end{proof}

\vskip5pt

\section{Application to 3-manifolds}

\vskip4pt

In this section we derive some applications of Theorem   \ref{HSdistance}.  
There are  other results in the recent literature about 3-manifolds $M$ with Heegaard splitting of large Hempel  distance,  for example  ~\cite{AS}, ~\cite{Ev} and \cite{MMS}. The advantage of our method, seems to us, is that it is very practical and allows in particular  derivation of  concrete lower bounds 
for the  distance of Heegaard splittings in combinatorial terms.

\subsection*{5.A. Application 1:  Heegaard diagrams} ${}^{}$

\vskip20pt

We first describe a practical way how to derive a  lower bound for 
the Hempel distance for a 
$3$-manifold $M$ given by a {\it Heegaard  diagram}: The latter is 
given by a standardly embedded handlebody $W \subset \reals^{3}$ with 
boundary surface $\Sigma = \partial W$, equipped with a complete 
decomposing system $\mathcal D \subset \Sigma$ that defines a second 
handlebody $V$ 
with $\Sigma =  \partial V$. The handlebody $V$ 
(which usually  cannot be embedded in $\reals^{3}$)
is determined by  the condition $\mathcal{D} \in \mathcal{CDS}(V)$, 
i.e. $V$ contains a system of 
essential disks with boundary $\mathcal D$.

\vskip10pt

One first picks at random a complete  decomposing system $\mathcal{E} 
\in \mathcal{CDS}(W)$. 
Then one modifies  the pair $(\mathcal{D, E})$ iteratively to find 
complete decomposing systems
$(\mathcal{D', E'})$ which bound  disks in  the same handlebodies $V$ 
and $W$ and, in addition, 
have no waves with respect to each other.  Since the sets 
$\mathcal{CDS}(V)$ and 
$\mathcal{CDS}(W)$  are countable, Hempel's existence  result for 
such systems given in 
Lemma \ref{quotehempel} implies that  there exists algorithms to find 
such systems $(\mathcal{D', E'})$.

\vskip10pt

For practical purposes,  an efficient procedure seems to be the  passage from the complete 
decomposing systems  $(\mathcal{D, E}) $ (by omitting some of the  curves from  $\mathcal{D}$ and 
$\mathcal{E}$) to {\it minimal decomposing systems} that cut the   surface $\Sigma$ into a single 
simply connected complementary  component,  i.e., a $2g$-punctured $2$-shere:  One can then apply 
{Whiteheads algorithm (see  ~\cite{Ka}  and the references given there)} to strictly reduce intersection  number until the minimum  is achieved. Adding the other disks back in, needed to make the minimal decomposing systems  complete, gives good candidates for systems $(\mathcal{D',  E'})$ that have 
no waves with respect to each other.

\vskip10pt

Once such systems $(\mathcal{D', E'})$ are found, one can easily  determine  a complete fat train track 
$\tau$  with $\mathcal{E}_{\tau} = \mathcal{E'}$   that carries  $\mathcal D'$, by considering the intersection arcs on  $\mathcal D'$  with the pair-of-pants on $\Sigma$ that are  complementary to the 
system $\mathcal E'$, compare Lemma \ref{fatcarries} and its proof. Next one  starts splitting 
$\tau_{0} = \tau$  iteratively until every unzipping path is  complete (or the weaker condition from Remark \ref{weaklyderived} is   satisfied), to obtain the (weakly) derived train track $\tau_{1}$ 
with  respect to $\mathcal D'$. One then splits again $\tau_{1}$, to  obtain 
a (weakly) derived train track $\tau_{2}$,  and so on, until the first  collision arises.  The last derived train track  $\tau_{n}$ constructed before the collision gives the desired   distance bound:

\vskip10pt

\begin{proposition} \label{Heegaarddiagram}
If a Heegaard splitting $V \cup_{\Sigma} W$  is given by a Heegaard  diagram as above, and  if the above construction yields a train track  $\tau_{n}$ that  carries $\mathcal{D}'$, then the Hempel distance of  the Heegaard  splitting satisfies: $$d(V, W) \geq n$$

\end{proposition}

\vskip10pt

\begin{proof}
This is an immediate consequence of the above construction and   Theorem \ref{HSdistance}.

\end{proof}

\vskip5pt

\subsection*{5.B. Application 2:  large distance via surgery} ${}^{}$

\vskip10pt

We now describe a practical way how to construct 3-manifolds with  Heegaard splittings of arbitrary high 
distance.

\vskip5pt

Let $\Sigma$ be a surface of genus $g \geq 2$, and let $\tau \subset  \Sigma$ be a complete fat train 
track  with associated complete decomposing system $\mathcal{E} =  \mathcal{E}_{\tau}$. One successively derives (or weakly derives)  train tracks   
$\tau_{0} = \tau, \tau_{1}, \tau_{2}, \ldots,  \tau_{n}$, and then chooses a curve $k \subset \Sigma$ that covers $\tau_{n}$. We note  that, since $k$ is carried by $\tau_{n}$ and hence by  $\tau_{0}$, it  does not contain waves with respect to $\mathcal E$.  Let $\delta_{k}$ denote the Dehn twist on $\Sigma$ along the  curve  $k$.

\vskip10pt

\begin{proposition} \label{firstsurgery}
The Heegaard splitting $V \cup_{\Sigma} W$, defined by the two  complete decomposing  systems 
$\mathcal E$ and $\mathcal{D} = \delta_{k}^m(\mathcal{E})$ on $\Sigma$  via the conditions 
$\mathcal{D} \in \mathcal{CDS}(V)$ and $\mathcal{E} \in  \mathcal{CDS}(W)$,  has distance 
$$d(V, W) \geq n\, ,$$ for any integer $m \in \mathbb{Z} \smallsetminus \{1, 0, -1\}$.

\end{proposition}

\vskip10pt

\begin{proof}
We first observe that for any $D \in \mathcal D$ and any intersection  point $P \in D \cap k$ the twisting number of $D$ along $k$ at $P$ satisfies $tw_{P}(D, k) = 0$. Hence one can apply Lemma 
\ref{twistingnumber} (a) directly to deduce that  $tw_{P}(\delta_{k}^m(D), k) =  |m|$, for any 
$m \in  \mathbb{Z}$. Thus Lemma \ref{twistingnumber} (b)  yields

that for every value of $m \in \mathbb{Z}  \smallsetminus \{1, 0, -1\}$ the curve system
$\mathcal{D} = \delta_{k}^m(\mathcal{E})$ is carried by $\tau_{n}$. In particular, it follows that  
$\mathcal D$ does not contain waves  with respect to  $\mathcal E$, and by symmetry of the construction  (i.e.  $\mathcal{E} = \delta_{k}^{-m}(\mathcal{D})$), the vice versa 
assertion is also true.  Hence Theorem \ref{HSdistance} applies to   give the claimed inequality.

\end{proof}

\vskip10pt

The above proof of Proposition \ref{firstsurgery} should be  compared to the proofs of both, 
Theorem 3.1  of  ~\cite{MMS} and Theorem 1.1 of ~\cite{Ev}.

\vskip10pt

It is a well known fact that changing the gluing map of a given   Heegaard splitting 
$M = V \cup_{\Sigma} W$ by an m-fold Dehn twist along a curve $k \subset \Sigma$ is  
equivalent to performing  $\frac{1}{m}$-Dehn surgery on $k$ along the slope determined by $\Sigma$ on a regular   neighborhood $N(k)$.   We call this a {\it  $\Sigma$-horizontal}  $\frac{1}{m}$-surgery on $k$ (compare \cite{LM}),  and the obtained manifold is denoted by   $M^\Sigma_{k}(\frac{1}{m})$ .

\vskip10pt

In the situation considered in Proposition  \ref{firstsurgery} the  manifold $M$ in question is  a connected sum of $g$ copies of  $S^{1}  \times  S^{2}$, provided with the standard Heegaard splitting of 
genus $g$. We will extend the construction in the next subsection to obtain an   analogous result for  knots $k$ in $S^{3}$.   Note, however, that in the case where $M$ is the connected sum of copies 
of  $S^{1} \times  S^{2}$, the alternative method described in \cite{MMS} does not  apply without additional modifications.

\vskip5pt

\subsection*{5.C. Application 3:  surgery on knots in $S^{3}$} ${}^{}$

\vskip10pt

Let $S^{3} = V \cup_{\Sigma} W$ be the standard Heegaard  decomposition  of the 
$3$-sphere of genus $g \geq 2$. Let $\mathcal{D} \in  \mathcal{CDS}(V)$ and 
$\mathcal{E} \in \mathcal{CDS}(W)$  be complete decomposing systems.
Let  $c \subset \Sigma$ be a  curve that has no waves with respect to  either $\mathcal D$ or  
$\mathcal E$, and which  fills every pair-of-pants $P$ complementary to either $\mathcal D$ or 
$\mathcal E$ (compare  Remark \ref{fillingpants}). Examples for such  $\mathcal{D,  E}$ and $c$ are not 
hard to find; see ~\cite{LM}  or \cite{MMS} for such examples with arbitrary large  
$g = \hbox{\rm genus} (\Sigma)$. We denote by $\tau_{\mathcal D}(c)$ and $\tau_{\mathcal E}(c)$ the 
two maximal train tracks that carry $c$ and have $\mathcal D$ and  $\mathcal E$ respectively as exceptional fibers.

\vskip5pt

We now consider an arbitrary minimal-maximal lamination $\mathcal{L}  \subset \Sigma$. 
We deduce from Lemma \ref{twistingnumber} and Remark  \ref{laminationversuscurve} (b)
that for all sufficiently large integers $t \geq 0$ the Dehn twist  $\delta_{c}^t$ gives a lamination 
$\mathcal{L}' = \delta_{c}^t(\mathcal{L})$ that covers both, $\tau_{\mathcal D}(c)$ and 
$\tau_{\mathcal E}(c)$. In particular, $\mathcal L'$  has no wave with respect to either $\mathcal D$ or 
$\mathcal{E}$.  

\vskip5pt

Next one derives successively train tracks  
$\tau_{0} =  \tau_{\mathcal E}(c), \tau_{1}, \tau_{2}, \ldots, \tau_{n}$ with respect  to  $\mathcal L'$ and then chooses a curve $k$ on $\Sigma$ that  covers both, $\tau_{n}$ and  $\tau_{\mathcal D}(c)$. Any curve  that is sufficiently close to $\mathcal L'$ (in the space of  projective measured laminations) will have this property. We note  that, since $k$ covers $\tau_{n}$  and hence also $\tau_{0}$, it fills every pair-of-pants complementary  to $\mathcal E$ and does not contain waves with respect to 
$\mathcal E$. Similarly, since $k$ covers $\tau_{\mathcal D}(c)$, it  fills every pair-of-pants complementary  to $\mathcal D$ and does also not contain waves with respect to  $\mathcal D$. 
Thus all conditions  of Proposition \ref{surgerylemma} are satisfied, so that the  resulting systems  
$\mathcal{D}^m = \delta_{k}^m(\mathcal{D})$ and $\mathcal{E}$ do not  have  waves with respect 
to each other, and $\mathcal{D}^m$ is carried by $\tau_{n}$. Thus we can apply Theorem 
\ref{HSdistance}.

\vskip5pt

It remains to recall the above observation that the $m$-fold Dehn twist  at $k$ amounts precisely to 
$\Sigma$-horizontal $\frac{1}{m}$-surgery on $k$.

\vskip5pt

Thus we have proved the following result, which has been stated in a  slightly weakened form at the beginning of the Introduction:

\begin{theorem} \label{secondsurgery}
For any integer $n \geq 0$ there exist a knot $k \in S^{3}$ and an integer $m_{0}$, such that for any  
$m \in \integers \smallsetminus \{m_{0}, m_{0}+1, m_{0} +   2, m_{0}+3\}$  the $3$-manifold
$M^\Sigma_{k}(\frac{1}{m})$, obtained by $\Sigma$-horizontal $\frac{1}{m}$-surgery on $k$,
admits a  Heegaard splitting $M^\Sigma_{k}(\frac{1}{m}) = V  \cup_{\Sigma} W$ which satisfies
$$d(V, W) \geq n \, .$$

\end{theorem}

\qed

\vskip5pt

We would like to comment on the relationship between this result and  Theorem 3.1 of \cite{MMS}: 
If a Heegaard splitting $M = V  \cup_{\Sigma} W$ is obtained through  Dehn filling along a curve 
$k \subset \Sigma \subset M$  from a Heegaard splitting of the knot exterior 
$E(k) = M \smallsetminus  \overset{\circ}{N}(k)$, then the distance of the latter splitting is bounded 
below by $d(V, W)$. The converse inequality, however, does not  hold. On the other hand, if $k$ is not primitive on either $V$ or  $W$ (as may well occur in our examples above),  then there is no Heegaard splitting of $E(k)$ that induces the  given splitting of $M$. Thus, although mathematically close, there is 
no direct implication either from Theorem 3.1 of \cite{MMS} to the  above Theorem \ref{secondsurgery}, nor conversely.


\vskip20pt

 \begin{thebibliography}{99}
 
 \vskip8pt
 
\bibitem{AS}
A. Abrams, S. Schleimer, \emph{ Distances of Heegaard splittings}, 
Geometric Topology, \textbf{9} (2005), 95 - 119.  

 
\bibitem{CB}
A.Casson, S. Bleiler \emph{Automorphisms of surfaces after Nielsen 
and Thurston},
Cambridge University Press, Cambridge, 1988.
 
\bibitem{CG}
A. Casson, C. Gordon, \emph{Reducing Heegaard splittings},  
Topology and  its Applications, \textbf{27} (3) (1987), 275 - 283. 

\bibitem{Ev}
T. Evans, \emph{High distance Heegaard splittings of $3$-manifolds},
Topology and its Applications,   \textbf{153} (14) (2006), 2631 - 
2647. 

 \bibitem{Har}
 W. Harvey, \emph{ Boundary structure of the modular group}. In 
Riemann 
surfaces and related topics: Proceedings of the 1978 Stony Brook 
Conference 
(State Univ. New York, Stony Brook, N.Y., 1978), pages 245 - 251, 
Princeton, 
N.J., 1981. Princeton Univ. Press. 

\bibitem{He} 
J. Hempel, \emph{3-manifolds as viewed from the curve complex},  
Topology \textbf{40} (3) (2001), 631 - 657.

\bibitem{Ka} 
T. Kaneto, \emph{On Genus 2 Heegaard Diagrams for the 3-Sphere},  
Trans. Amer. Math. Soc. \textbf{276} (2) (1983), 583 - 597.

\bibitem{Kl}
E. Klarreich, \emph{The boundary at infinity of the curve complex and 
the relative Teichmuller space}, Preprint. 

\bibitem{Ko}
T. Kobayashi, \emph{Heights of simple loops and pseudo-Anosov 
homeomor- 
phisms}. In Braids (Santa Cruz, CA, 1986), pages 327 - 338. Amer. 
Math. Soc., 
Providence, RI, 1988.

\bibitem{LM}
M. Lustig, Y. Moriah, \emph{A finiteness result for Heegaard  
splitting}, 
Topology \textbf{43} (2004), 1165  - 1182.

\bibitem{MM}
H. Masur, Y. Minsky, {\it    Geometry of the complex of curves I: 
hyperbolicity }, 
Inventiones. Mathematicae \textbf{138} (1) (1999), 103 - 149. 

\bibitem{MM1}
H. Masur, Y. Minsky, \emph{Geometry of the complex of curves II: 
hierarchical structure}, Geom. Funct. Anal. \textbf{10} (4) (2000),  
902 - 974. 

\bibitem{MMS}
Y. Minsky, Y. Moriah, S. Schleimer, \emph{High distance knots}, 
to appear in AGT. arXiv:math.GT/0607265. 

\bibitem{Sc}
S. Schleimer, \emph{Notes on the  complex of curves}, Preprint.\\
http://www.math.rutgers.edu/ $\sim$saulsch/Maths/notes.pdf


\end{thebibliography}
\end{document}